\documentclass[12pt]{amsart}
\usepackage{graphicx}
\usepackage{amsmath, amscd, amssymb, amsthm}
\usepackage[subnum]{cases}
\usepackage{changepage}
\usepackage{xcolor}
\usepackage{marginnote}
\usepackage{hyperref}
\usepackage{cleveref}
\usepackage[all]{xy}
\usepackage{tabularx}
\usepackage{ltablex}
\usepackage{longtable}
\usepackage[T1]{fontenc}

\newtheorem{theorem}{Theorem}[section]
\newtheorem{lemma}[theorem]{Lemma}

\newtheorem{proposition}[theorem]{Proposition}

\newtheorem{conjecture}[theorem]{Conjecture}

\theoremstyle{definition}
\newtheorem{definition}[theorem]{Definition}
\newtheorem{remark}[theorem]{Remark}
\numberwithin{equation}{section}
\newtheorem{example}[theorem]{Example}

\newtheorem{setting}[theorem]{Setting}

\usepackage{fancyhdr}
\begin{document}

\normalfont

\title{Topics on Geometric and Representation Theoretic Aspects of Period Rings I}
\author{Xin Tong}

\maketitle

\begin{abstract}
\rm We consider more general framework than the corresponding one considered in our previous work on the Hodge-Iwasawa theory. In our current consideration we consider the corresponding more general base spaces, namely the analytic adic spaces and analytic perfectoid spaces in Kedlaya's AWS Lecture notes. We hope our discussion will also shed some light on further generalization to even more general spaces such as those considered by Gabber-Ramero namely one just considers certain topological rings which satisfy the Fontaine-Wintenberger idempotent correspondence and calls them perfectoid generalizing the notions from Scholze, Fontaine, Kedlaya-Liu and Kedlaya (AWS Lecture notes). Actually some of the discussion we presented here is already in some more general form for this purpose (although we have not made enough efforts to write all the things). 
\end{abstract}

\footnotetext[1]{Version: Feb 28 2021.}

\newpage

\tableofcontents

\newpage

\section{Introduction}

\subsection{General Perfectoids and Preperfectoids}

\noindent Our previous work on Hodge-Iwasawa theory aimed at the corresponding deformation of the corresponding Hodge structure after \cite{KL1} and \cite{KL2}. The corresponding application in our mind targets the corresponding noncommutative Iwasawa theory and the corresponding $p$-adic local systems with general Banach coefficients. The corresponding background of the foundations on the analytic geometry comes from essentially \cite{KL1} where Banach analog of Fontaine's perfectoid spaces in the Tate case was defined and studied extensively. \\

\indent The corresponding analytic Huber analog of the main results of \cite{KL1} was initiated in \cite{Ked1} where the analytic Huber analog of perfectoids are defined, and the analytic Huber analog of the corresponding perfectoid correspondence is established and the corresponding descent for vector bundles and stably-pseudocoherent sheaves is established in the analytic topology.\\

\indent In our current development of the corresponding Hodge-Iwasawa theory we consider the corresponding generality at least parallel to \cite{Ked1}. We now consider the corresponding analytic analog of our previous consideration which generalized the corresponding \cite{KL1} and \cite{KL2} to the corresponding big coefficient situation. Without deformation in our sense \cite{KL1} and \cite{KL2} have already considered the corresponding free of trivial norm cases in certain context and considerations. \\

\indent Certainly inverting $p$ takes back us to the Tate situation, but what we know is that this is at least not included in our previous consideration, since the relative Robba rings are not the same (even in our current situation many are still Tate as their owns). So we have to basically establish the corresponding parallel theory, in order to treat the situation where the base spaces or rings are not Tate. And more importantly the integral Robba rings are not Tate, which needs to be discussed in order to apply to the situation where the base spaces or rings are not Tate. \\

\indent With the corresponding notations in Chapter 5, we have:

\begin{theorem}
Over $A$-relative preperfectoid Robba rings $\widetilde{\Pi}_{?,R,A}$ ($?=\empty,r,\infty,[r_1,r_2],\{[s,r]\}$, $0<r_1\leq r_2/p^{hk}$)  constructed over analytic base $(R,R^+)$, we have the equivalence among the categories of the Frobenius \'etale-stably-pseudocoherent modules, which could be further compared in equivalence to the the pseudocoherent sheaves over adic Fargues-Fontaine curves in both \'etale and pro-\'etale topology.	
\end{theorem}

\indent This is basically the corresponding commutative version where we encode the corresponding ring $A$ into the corresponding Huber spectrum in \cite{KL1} and \cite{KL2}. The corresponding noncommutative version is also true but we have to do this slightly different by deforming the structure sheaves.  With the corresponding notations again in Chapter 5, we have:

\begin{theorem}
Over $B$-relative preperfectoid Robba rings $\widetilde{\Pi}_{?,R,B}$ ($?=\empty,r,\infty,[r_1,r_2],\{[s,r]\}$, $0<r_1\leq r_2/p^{hk}$) constructed over analytic base $(R,R^+)$, we have the equivalence among the categories of the Frobenius $B$-\'etale-stably-pseudocoherent modules, which could be further compared in equivalence to the the $B$-pseudocoherent sheaves over adic Fargues-Fontaine curves in \'etale topology.	
\end{theorem}

\indent Our ultimate consideration will after \cite{GR}, where $(R,R^+)$ is not necessarily always assumed to be analytic, especially in the integral setting this is quite general. With the corresponding notations again in Chapter 6, we have:

\begin{theorem}
Over $A$-relative preperfectoid Robba rings $\widetilde{\Pi}_{?,R,A}$ ($?=[r_1,r_2],\{[s,r]\}$, $0<r_1\leq r_2/p^{hk}$) constructed over analytic base $(R,R^+)$, we have the equivalence of the categories of the Frobenius \'etale-stably-pseudocoherent modules, which could be further compared in equivalence to the the pseudocoherent sheaves over adic Fargues-Fontaine curves in the \'etale or pro-\'etale topology.	
\end{theorem}

\begin{theorem}
Over $B$-relative preperfectoid Robba rings $\widetilde{\Pi}_{?,R,B}$ ($?=[r_1,r_2],\{[s,r]\}$, $0<r_1\leq r_2/p^{hk}$) constructed over analytic base $(R,R^+)$, we have the equivalence of the categories of the Frobenius $B$-stably-pseudocoherent modules.	
\end{theorem}

\subsection{Some Consideration in the Future}

More thorough applications of \cite{Lu2} will be expected in our study on the corresponding derived algebraic geometry of the corresponding period rings\footnote{The original motivation comes from the study of the algebraic pseudocoherent sheaves over algebraic Fargues-Fontaine curves originally in \cite{KL2}, and also in \cite{XT1} and \cite{XT2} where the corresponding functional analytic information around locally convex vector spaces (after Bourbaki) is not considered.}. We only touched this to some very transparent extent which will be somehow reflecting some further consideration along the corresponding applications in our mind.\\

New robust developments from Clausen-Scholze \cite{CS}, Bambozzi-Ben-Bassat-Kremnizer \cite{BBBK} and Bambozzi-Kremnizer \cite{BK1} may shed lights on the derived Galois (possibly in more general setting for fundamental groups) deformation theory (and some unobstructed crystalline substacks) such as in \cite{GV} and possibly (we conjecture) the corresponding possible theory of some derived eigenvarieties and derived Selmer varieties in Kim's nonabelian Chabauty. We hope to give more hard thorough study and understanding around this. Especially in our situation we will come across many technical issues where these new developments could help.

\


\section{Period Rings in General Setting with General Coefficients}

\noindent In this section we start by defining Kedlaya-Liu's style period rings with coefficients in Tate adic Banach rings which are of finite type.

\begin{setting}
We consider the corresponding setting up which is as in the following. First we consider a field $E$ which is analytic nonarchimedean with normalized norm such that the corresponding uniformizer $\pi_E$ is of norm $1/p$. And we assume that the corresponding residue field of the field $E$ takes the form of $\mathbb{F}_{p^h}$. Then we fix a corresponding uniform adic Banach algebra $(R,R^+)$ over $\mathbb{F}_{p^h}$ which is not required to be contain topological nilpotent unit. However we need to assume that that this is analytic in the sense of \cite{KL1} and \cite[Definition 1.1.2]{Ked1}. We recall that this means the set of the corresponding topological nilpotents generates the unit ideal. And let $A$ be a general rigid analytic affinoid over $\mathbb{Q}_p$ or over $\mathbb{F}_p((t))$.
\end{setting}

\indent Then in this generality we define our large coefficient Robba rings following Kedlaya-Liu in the Tate algebra situation namely we have $A=\mathbb{Q}_p\{X_1,...,X_d\}$ or $A=\mathbb{F}_p((t))\{X_1,...,X_d\}$:


\begin{definition}\mbox{\bf{(After Kedlaya-Liu \cite[Definition 4.1.1]{KL2})}} Now consider the following constructions. First we consider the corresponding Witt vectors coming from the corresponding adic ring $(R,R^+)$. First we consider the corresponding generalized Witt vectors with respect to $(R,R^+)$ with the corresponding coefficients in the Tate algebra with the general notation $W(R^+)[[R]]$. The general form of any element in such deformed ring could be written as $\sum_{i\geq 0,i_1\geq 0,...,i_d\geq 0}\pi^i[\overline{y}_i]X_1^{i_1}...X_d^{i_d}$. Then we take the corresponding completion with respect to the following norm for some radius $t>0$:
\begin{align}
\|.\|_{t,A}(\sum_{i\geq 0,i_1\geq 0,...,i_d\geq 0}\pi^i[\overline{y}_i]X_1^{i_1}...X_d^{i_d}):= \max_{i\geq 0,i_1\geq 0,...,i_d\geq 0}p^{-i}\|.\|_R(\overline{y}_i)	
\end{align}
which will give us the corresponding ring $\widetilde{\Pi}_{\mathrm{int},t,R,A}$ such that we could put furthermore that:
\begin{align}
\widetilde{\Pi}_{\mathrm{int},R,A}:=\bigcup_{t>0} \widetilde{\Pi}_{\mathrm{int},t,R,A}.	
\end{align}
Then as in \cite[Definition 4.1.1]{KL2}, we now put the ring $\widetilde{\Pi}_{\mathrm{bd},t,R,A}:=\widetilde{\Pi}_{\mathrm{int},t,R,A}[1/\pi]$ and we set:
\begin{align}
\widetilde{\Pi}_{\mathrm{bd},R,A}:=\bigcup_{t>0} \widetilde{\Pi}_{\mathrm{bd},t,R,A}.	
\end{align}
The corresponding Robba rings with respect to some intervals and some radius could be defined in the same way as in \cite[Definition 4.1.1]{KL2}. To be more precise we consider the completion of the corresponding ring $W(R^+)[[R]][1/\pi]$ with respect to the following norm for some $t>0$ where $t$ lives in some prescribed interval $I=[s,r]$: 
\begin{align}
\|.\|_{t,A}(\sum_{i,i_1\geq 0,...,i_d\geq 0}\pi^i[\overline{y}_i]X_1^{i_1}...X_d^{i_d}):= \max_{i\geq 0,i_1\geq 0,...,i_d\geq 0}p^{-i}\|.\|_R(\overline{y}_i).	
\end{align}
This process will produce the corresponding Robba rings with respect to  the given interval $I=[s,r]$. Now for particular sorts of intervals $(0,r]$ we will have the corresponding Robba ring $\widetilde{\Pi}_{r,R,A}$ and we will have the corresponding Robba ring $\widetilde{\Pi}_{\infty,R,A}$	if the corresponding interval is taken to be $(0,\infty)$. Then in our situation we could just take the corresponding union throughout all the radius $r>0$ to define the corresponding full Robba ring taking the notation of $\widetilde{\Pi}_{R,A}$.
\end{definition}

\begin{remark}
The corresponding Robba rings $\widetilde{\Pi}_{\mathrm{bd},R,A}$, $\widetilde{\Pi}_{R,A}$, $\widetilde{\Pi}_{I,R,A}$, $\widetilde{\Pi}_{r,R,A}$, $\widetilde{\Pi}_{\infty,R,A}$ are actually themselves Tate adic Banach rings. However in many further application the non-Tateness of the ring $R$ will cause some reason for us to do the corresponding modification, which is considered on this level in fact in \cite{KL1} in the context therein.	
\end{remark}

\begin{definition}
Then for any general affinoid algebra $A$ over the corresponding base analytic field, we just take the corresponding quotients of the corresponding rings defined in the previous definition over some Tate algebras in rigid analytic geometry, with the same notations though $A$ now is more general. Note that one can actually show that the definition does not depend on the corresponding choice of the corresponding presentations over $A$.
\end{definition}

\begin{remark}
Again in this situation more generally, the corresponding Robba rings $\widetilde{\Pi}_{\mathrm{bd},R,A}$, $\widetilde{\Pi}_{R,A}$, $\widetilde{\Pi}_{I,R,A}$, $\widetilde{\Pi}_{r,R,A}$, $\widetilde{\Pi}_{\infty,R,A}$ are actually themselves Tate adic Banach rings.	
\end{remark}

\indent Note that we can also as in the situation of \cite{KL2} and \cite{XT2} consider the corresponding the corresponding property checking of the corresponding period rings defined aboave. We collect the corresponding statements here while the the proof could be found in \cite{XT2}:

\begin{lemma} \mbox{\bf{(After Kedlaya-Liu \cite[Lemma 5.2.6]{KL2})}}
For any two radii $0<r_1<r_2$ we have the corresponding equality:
\begin{align}
\widetilde{\Pi}_{\mathrm{int},r_2,R,\mathbb{Q}_p\{T_1,...,T_d\}}\bigcap \widetilde{\Pi}_{[r_1,r_2],R,\mathbb{Q}_p\{T_1,...,T_d\}}	=\widetilde{\Pi}_{\mathrm{int},r_1,R,\mathbb{Q}_p\{T_1,...,T_d\}}.
\end{align}

\end{lemma}

\begin{proof}
See \cite[Lemma 5.2.6]{KL2} and \cite[Proposition 2.13]{XT2}.	
\end{proof}

\begin{lemma} \mbox{\bf{(After Kedlaya-Liu \cite[Lemma 5.2.6]{KL2})}}
For any two radii $0<r_1<r_2$ we have the corresponding equality:
\begin{align}
\widetilde{\Pi}_{\mathrm{int},r_2,R,\mathbb{F}_p((t))\{T_1,...,T_d\}}\bigcap \widetilde{\Pi}_{[r_1,r_2],R,\mathbb{F}_p((t))\{T_1,...,T_d\}}	=\widetilde{\Pi}_{\mathrm{int},r_1,R,\mathbb{F}_p((t))\{T_1,...,T_d\}}.
\end{align}

\end{lemma}

\begin{proof}
See \cite[Lemma 5.2.6]{KL2} and \cite[Proposition 2.13]{XT2}.	
\end{proof}

\begin{lemma} \mbox{\bf{(After Kedlaya-Liu \cite[Lemma 5.2.6]{KL2})}}
For general affinoid $A$ as above (over $\mathbb{Q}_p$ or $\mathbb{F}_p((t))$) and for any two radii $0<r_1<r_2$ we have the corresponding equality:
\begin{align}
\widetilde{\Pi}_{\mathrm{int},r_2,R,A}\bigcap \widetilde{\Pi}_{[r_1,r_2],R,A}	=\widetilde{\Pi}_{\mathrm{int},r_1,R,A}.
\end{align}

\end{lemma}

\begin{proof}
See \cite[Lemma 5.2.6]{KL2} and \cite[Proposition 2.14]{XT2}.	
\end{proof}

\begin{lemma} \mbox{\bf{(After Kedlaya-Liu \cite[Lemma 5.2.10]{KL2})}}
For any four radii $0<r_1<r_2<r_3<r_4$ we have the corresponding equality:
\begin{align}
\widetilde{\Pi}_{[r_1,r_3],R,\mathbb{Q}_p\{T_1,...,T_d\}}\bigcap \widetilde{\Pi}_{[r_2,r_4],R,\mathbb{Q}_p\{T_1,...,T_d\}}	=\widetilde{\Pi}_{[r_1,r_4],R,\mathbb{Q}_p\{T_1,...,T_d\}}.
\end{align}

\end{lemma}

\begin{proof}
See \cite[Lemma 5.2.10]{KL2} and \cite[Proposition 2.16]{XT2}.	
\end{proof}

\begin{lemma} \mbox{\bf{(After Kedlaya-Liu \cite[Lemma 5.2.10]{KL2})}}
For any four radii $0<r_1<r_2<r_3<r_4$ we have the corresponding equality:
\begin{align}
\widetilde{\Pi}_{[r_1,r_3],R,\mathbb{F}_p((t))\{T_1,...,T_d\}}\bigcap \widetilde{\Pi}_{[r_2,r_4],R,\mathbb{F}_p((t))\{T_1,...,T_d\}}	=\widetilde{\Pi}_{[r_1,r_4],R,\mathbb{F}_p((t))\{T_1,...,T_d\}}.
\end{align}

\end{lemma}

\begin{proof}
See \cite[Lemma 5.2.10]{KL2} and \cite[Proposition 2.16]{XT2}.	
\end{proof}

\begin{lemma} \mbox{\bf{(After Kedlaya-Liu \cite[Lemma 5.2.10]{KL2})}}
For any four radii $0<r_1<r_2<r_3<r_4$ we have the corresponding equality:
\begin{align}
\widetilde{\Pi}_{[r_1,r_3],R,A}\bigcap \widetilde{\Pi}_{[r_2,r_4],R,A}	=\widetilde{\Pi}_{[r_1,r_4],R,A}.
\end{align}

\end{lemma}

\begin{proof}
See \cite[Lemma 5.2.10]{KL2} and \cite[Proposition 2.17]{XT2}.	
\end{proof}


\


\section{The Frobenius Modules and Frobenius Sheaves over the Period Rings} \label{chapter3}
 
\noindent We now consider the corresponding Frobenius actions and the corresponding action reflected on the corresponding Hodge-Iwasawa modules over general analytic $R$. As in the Tate situation we consider the corresponding 'arithmetic' Frobenius as well.

\begin{setting}
We are going to use the corresponding notation $F$ to denote the corresponding relative Frobenius induced from the corresponding $p^h$-power absolute Frobenius from $R$. This would mean that when we consider the corresponding Frobenius up to higher order.
\end{setting}

\indent Furthermore in our situation we have the corresponding sheaves of period rings (carrying the coefficient $A$) over $\mathrm{Spa}(R,R^+)$ in the corresponding pro-\'etale site:
\begin{align}
\widetilde{\Pi}_{\mathrm{bd},\mathrm{Spa}(R,R^+)_\text{pro\'et},A}, \widetilde{\Pi}_{\mathrm{Spa}(R,R^+)_\text{pro\'et},A}, \widetilde{\Pi}_{I,\mathrm{Spa}(R,R^+)_\text{pro\'et},A}, \widetilde{\Pi}_{r,\mathrm{Spa}(R,R^+)_\text{pro\'et},A}, \widetilde{\Pi}_{\infty,\mathrm{Spa}(R,R^+)_\text{pro\'et},A}	
\end{align}
which we will also briefly write as:
\begin{align}
\widetilde{\Pi}_{\mathrm{bd},*,A}, \widetilde{\Pi}_{*,A}, \widetilde{\Pi}_{I,*,A}, \widetilde{\Pi}_{r,*,A}, \widetilde{\Pi}_{\infty,*,A}	
\end{align}
when the corresponding topology is basically clear to the readers (one can also consider other kinds of topology in the same fashion such as the corresponding \'etale situation in our current context).

\begin{definition} \mbox{\bf{(After Kedlaya-Liu \cite[Definition 4.4.4]{KL2})}} The corresponding Frobenius modules over the corresponding period sheaves will be defined to be finite projective modules carrying the semilinear action from $F^k$ (where we allow some power $k>0$) with some pullback isomorphic requirement which could be defined in the following way. Over the rings without intervals or radius this will mean that we have $F^{k*}M\overset{\sim}{\rightarrow}M$. While over the corresponding Robba rings with radius $r>0$ this will mean that we have $F^{k*}M\otimes_{*_{r/p^{hk}}} *_{r/p^{hk}}\overset{\sim}{\rightarrow}M\otimes_{*_r}*_{r/p^{hk}}$. While over the corresponding Robba rings with interval $[r_1,r_2]$ this will mean that we have $F^{k*}M\otimes_{*_{[r_1/p^{hk},r_2/p^{hk}]}} *_{[r_1,r_2/p^{hk}]}\overset{\sim}{\rightarrow}M\otimes_{*_{[r_1,r_2]}}*_{[r_1,r_2/p^{hk}]}$. Over the corresponding Robba rings without any intervals or radii we assume that they are base change from some modules over the corresponding Robba rings with some radii. 
	
\end{definition}


\begin{definition} \mbox{\bf{(After Kedlaya-Liu \cite[Definition 4.4.4]{KL2})}} The corresponding pseudocoherent or fpd Frobenius modules over the corresponding period sheaves will be defined to be pseudocoherent or fpd modules carrying the semilinear action from $F^k$ (where we allow some power $k>0$) with some pullback isomorphic requirement which could be defined in the following way. Over the rings without intervals or radius this will mean that we have $F^{k*}M\overset{\sim}{\rightarrow}M$. While over the corresponding Robba rings with radius $r>0$ this will mean that we have $F^{k*}M\otimes_{*_{r/p^{hk}}} *_{r/p^{hk}}\overset{\sim}{\rightarrow}M\otimes_{*_r}*_{r/p^{hk}}$. While over the corresponding Robba rings with interval $[r_1,r_2]$ this will mean that we have $F^{k*}M\otimes_{*_{[r_1/p^{hk},r_2/p^{hk}]}} *_{[r_1,r_2/p^{hk}]}\overset{\sim}{\rightarrow}M\otimes_{*_{[r_1,r_2]}}*_{[r_1,r_2/p^{hk}]}$. Over the corresponding Robba rings without any intervals or radii we assume that they are base changes from some modules over the corresponding Robba rings with some radii. And then all the modules are required to be complete with respect to the natural topology over any perfectoid subdomain within the corresponding pro-\'etale topology, and the corresponding modules over the corresponding Robba rings with respect to some radius $r>0$ will basically be required to be base change to any \'etale-stably pseudocoherent (over any specific chosen perfectoid subdomain) modules (note that this will include the corresponding glueing along the space $A$ with respect to the corresponding implicit \'etale topology induced over $A$). 
	
\end{definition}

\

\indent Then as in \cite[Definition 4.4.4]{KL2} we consider the corresponding Frobenius modules over the period rings instead of period sheaves (and we do not either put topological conditions, which will be added later in specific consideration):


\begin{definition} \mbox{\bf{(After Kedlaya-Liu \cite[Definition 4.4.4]{KL2})}} The corresponding Frobenius modules over the corresponding period rings will be defined to be finite projective modules carrying the semilinear action from $F^k$ (where we allow some power $k>0$) with some pullback isomorphic requirement which could be defined in the following way. Over the rings without intervals or radius this will mean that we have $F^{k*}M\overset{\sim}{\rightarrow}M$. While over the corresponding Robba rings with radius $r>0$ this will mean that we have $F^{k*}M\otimes_{*_{r/p^{hk}}} *_{r/p^{hk}}\overset{\sim}{\rightarrow}M\otimes_{*_r}*_{r/p^{hk}}$. While over the corresponding Robba rings with interval $[r_1,r_2]$ this will mean that we have $F^{k*}M\otimes_{*_{[r_1/p^{hk},r_2/p^{hk}]}} *_{[r_1,r_2/p^{hk}]}\overset{\sim}{\rightarrow}M\otimes_{*_{[r_1,r_2]}}*_{[r_1,r_2/p^{hk}]}$. Over the corresponding Robba rings without any intervals or radii we assume that they are base change from some modules over the corresponding Robba rings with some radii. 
	
\end{definition}

\begin{remark}
As in \cite[Definition 4.4.4]{KL2} we did not out any topological conditions on the ring theoretic and algebraic representation theoretic objects defined above, but this will be definitely more precise in later development. Certainly this is more relevant with respect to the following definition.	
\end{remark}

\begin{definition} \mbox{\bf{(After Kedlaya-Liu \cite[Definition 4.4.4]{KL2})}} The corresponding pseudocoherent or fpd Frobenius modules over the corresponding period rings will be defined to be pseudocoherent or fpd modules carrying the semilinear action from $F^k$ (where we allow some power $k>0$) with some pullback isomorphic requirement which could be defined in the following way. Over the rings without intervals or radius this will mean that we have $F^{k*}M\overset{\sim}{\rightarrow}M$. While over the corresponding Robba rings with radius $r>0$ this will mean that we have $F^{k*}M\otimes_{*_{r/p^{hk}}} *_{r/p^{hk}}\overset{\sim}{\rightarrow}M\otimes_{*_r}*_{r/p^{hk}}$. While over the corresponding Robba rings with interval $[r_1,r_2]$ this will mean that we have $F^{k*}M\otimes_{*_{[r_1/p^{hk},r_2/p^{hk}]}} *_{[r_1,r_2/p^{hk}]}\overset{\sim}{\rightarrow}M\otimes_{*_{[r_1,r_2]}}*_{[r_1,r_2/p^{hk}]}$. Over the corresponding Robba rings without any intervals or radii we assume that they are base change from some modules over the corresponding Robba rings with some radii. 
	
\end{definition}

\indent Then we define over our current analytic base the corresponding pseudocoherent and fpd sheaves carrying the corresponding Frobenius modules.

\begin{definition}\mbox{\bf{(After Kedlaya-Liu \cite[Definition 4.4.6]{KL2})}}
Over the Robba rings with respect to the corresponding all finite intervals contained in $(0,\infty)$, we consider a family $(M_{[s,r]})_{[s,r]}$ of finite projective modules with respect to the corresponding intervals. Then we call this family a finite projective $F^k$-bundle over $\widetilde{\Pi}_{R,A}$ if each member in the family is assumed to be a corresponding finite projective $F^k$-modules.	
\end{definition}

\begin{definition}\mbox{\bf{(After Kedlaya-Liu \cite[Definition 4.4.6]{KL2})}}
Over the Robba rings with respect to the corresponding all finite intervals contained in $(0,\infty)$, we consider a family $(M_{[s,r]})_{[s,r]}$ of fpd or pseudocoherent modules with respect to the corresponding intervals. Then we call this family a fpd or pseudocoherent $F^k$-sheaf over $\widetilde{\Pi}_{R,A}$ if each member in the family is assumed to be a corresponding stably or \'etale-stably fpd or pseudocoherent $F^k$-modules.	
\end{definition}

\begin{proposition} \mbox{\bf{(After Kedlaya-Liu \cite[Lemma 4.4.8]{KL2})}}
Any pseudocoherent $F^k$-modules over the corresponding period rings above admits a surjective morphism from some finite projective $F^k$-modules over the same types of period rings.	
\end{proposition}

\begin{proof}
This is the corresponding relative and analytic version of the corresponding \cite[Lemma 4.4.8]{KL2}. The proof is parallel, see \cite[Lemma 4.4.8]{KL2}.	
\end{proof}

\


\section{Comparison in the \'Etale Setting}

\noindent We now consider the corresponding parallel consideration to \cite[Section 4.5]{KL2}. Essentially in our current context, this is really the corresponding purely analytic setting especially when we do not have the chance to invert $\pi$. Certainly we will also consider the corresponding deformed setting. The corresponding space we are going to work on is just the corresponding perfectoid spaces (in the corresponding analytic situation) taking the corresponding general form of $\mathrm{Spa}(R,R^+)$, and consider any topological ring $Z$ and the corresponding sheaf $\underline{Z}$. We will consider the corresponding finite projective, pseudocoherent and finite projective dimension local systems of $\underline{Z}$-modules. Note that again in this analytic setting the corresponding definitions of such modules are locally of the corresponding same types. This means in particular in the situation of finite projective situation the corresponding definition is locally the finite projective sheaves of modules over $\underline{Z}$ instead of the corresponding naive locally finite free ones especially in our large coefficient situation.

\begin{lemma} \mbox{\bf{(After Kedlaya-Liu \cite[Lemma 4.5.3]{KL2})}}
The corresponding taking element to the corresponding element after the action of the corresponding operator $F-1$ will be to realize the following exact sequences as in \cite[Lemma 4.5.3]{KL2}:
\[
\xymatrix@C+0pc@R+0pc{
0 \ar[r] \ar[r] \ar[r] &\underline{\mathcal{O}_{E_k}}  \ar[r] \ar[r] \ar[r] &\widetilde{\Omega}_{\mathrm{int},R} \ar[r] \ar[r] \ar[r] &\widetilde{\Omega}_{\mathrm{int},R} \ar[r] \ar[r] \ar[r] &0,\\
0 \ar[r] \ar[r] \ar[r] &\underline{{E}_{k}}  \ar[r] \ar[r] \ar[r] &\widetilde{\Omega}_{R} \ar[r] \ar[r] \ar[r] &\widetilde{\Omega}_{R} \ar[r] \ar[r] \ar[r] &0,\\
0 \ar[r] \ar[r] \ar[r] &\underline{\mathcal{O}_{E_k}}  \ar[r] \ar[r] \ar[r] &\widetilde{\Pi}_{\mathrm{int},R} \ar[r] \ar[r] \ar[r] &\widetilde{\Pi}_{\mathrm{int},R} \ar[r] \ar[r] \ar[r] &0,\\
0 \ar[r] \ar[r] \ar[r] &\underline{{E_k}}  \ar[r] \ar[r] \ar[r] &\widetilde{\Pi}_{R} \ar[r] \ar[r] \ar[r] &\widetilde{\Pi}_{R} \ar[r] \ar[r] \ar[r] &0.\\
}
\]	
\end{lemma}

\begin{proof}
See \cite[Lemma 4.5.3]{KL2}.	
\end{proof}

\begin{proposition} \mbox{\bf{(After Kedlaya-Liu \cite[Lemma 4.5.4]{KL2})}}
We have the corresponding strictly pseudocoherence of any corresponding finitely presented modules over $\widetilde{\Omega}_{\mathrm{int},R}$ or $\widetilde{\Pi}_{\mathrm{int},R}$. Therefore we do not have to consider the corresponding topological issue and the stability issue further in this certain special situation.	
\end{proposition}

\begin{proof}
See \cite[Lemma 4.5.4]{KL2}.	
\end{proof}

\indent The following result is then achievable:

\begin{proposition}\mbox{\bf{(After Kedlaya-Liu \cite[Theorem 4.5.7]{KL2})}}
Consider the following categories. The first one is the corresponding category of all the finite projective or pseudocoherent $\underline{\mathcal{O}_{E_k}}$-local systems. The second one is the corresponding category of all the finite projective or stably-pseudocoherent $F^k$-$\widetilde{\Omega}_{\mathrm{int},R}$-modules. The third one is the corresponding category of all the finite projective or stably-pseudocoherent $F^k$-$\widetilde{\Pi}_{\mathrm{int},R}$-modules. The fourth one is the corresponding category of all the finite projective or stably-pseudocoherent $F^k$-$\widetilde{\Omega}_{\mathrm{int},*}$-sheaves. The last one is the corresponding category of all the finite projective or stably-pseudocoherent $F^k$-$\widetilde{\Pi}_{\mathrm{int},*}$-sheaves. Here we consider the corresponding analytic topology. \\
\indent Then we have that these categories are actually equivalent. Similar parallel statement could then be made to fpd objects.	
\end{proposition}

\begin{remark}
The corresponding stably-pseudocoherent indication in the previous proposition is actually not serious as in \cite[Lemma 4.5.7]{KL2}, the reason is that by the previous proposition, everything is already complete with respect to the natural topology.	
\end{remark}

\begin{proof}
This is just the corresponding analytic analog of the corresponding result in \cite[Lemma 4.5.7]{KL2}. We just briefly mention the corresponding functors which realize the corresponding equivalence. The corresponding one from the first category to the corresponding sheaves is naturally just the corresponding base change. The corresponding functor from the sheaves to the corresponding modules is naturally the corresponding global section functor after the corresponding foundation in \cite[Theorem 1.4.2, Theorem 1.4.18]{Ked1}. 
\end{proof}

\begin{remark}
Since we touched the corresponding foundations in \cite[Theorem 1.4.2, Theorem 1.4.18]{Ked1} on the corresponding glueing of the vector bundles and the corresponding stably-pseudocoherent modules in the situation of analytic topology. However the corresponding descent in the \'etale, pro-\'etale and $v$-topology situations are actually parallel to the corresponding \cite[Theorem 2.5.5, Theorem 2.5.14, Theorem 3.4.8, Theorem 3.5.8]{KL2}, which will basically upgrade naturally this proposition to the \'etale, pro-\'etale and $v$-topology situations, although we have not established the chance to present. 
\end{remark}

\indent Now carrying the corresponding coefficient $A$ we could have at least the corresponding fully faithfulness.

\begin{proposition}\mbox{\bf{(After Kedlaya-Liu \cite[Theorem 4.5.7]{KL2})}}
Consider the following categories. The first one is the corresponding category of all the finite projective or pseudocoherent $\underline{\mathcal{O}_{{E_k},A}}$-local systems. The second one is the corresponding category of all the finite projective or stably-pseudocoherent $F^k$-$\widetilde{\Omega}_{\mathrm{int},R,A}$-modules. The third one is the corresponding category of all the finite projective or stably-pseudocoherent $F^k$-$\widetilde{\Pi}_{\mathrm{int},R,A}$-modules. The fourth one is the corresponding category of all the finite projective or stably-pseudocoherent $F^k$-$\widetilde{\Omega}_{\mathrm{int},*,A}$-sheaves. The last one is the corresponding category of all the finite projective or stably-pseudocoherent $F^k$-$\widetilde{\Pi}_{\mathrm{int},*,A}$-sheaves. Here we consider the corresponding analytic topology. \\
\indent Then we have that the first category could be embedded fully faithfully into the rest four categories and the corresponding modules could be compared to the sheaves under equivalences of the corresponding categories as long as we assume the corresponding sousperfectoidness of the ring $A$ which preserves the corresponding exactness of the exact sequences:
\[
\xymatrix@C+0pc@R+0pc{
0 \ar[r] \ar[r] \ar[r] &\underline{\mathcal{O}_{E_k}}  \ar[r] \ar[r] \ar[r] &\widetilde{\Omega}_{\mathrm{int},R} \ar[r] \ar[r] \ar[r] &\widetilde{\Omega}_{\mathrm{int},R} \ar[r] \ar[r] \ar[r] &0,\\
0 \ar[r] \ar[r] \ar[r] &\underline{{E}_{k}}  \ar[r] \ar[r] \ar[r] &\widetilde{\Omega}_{R} \ar[r] \ar[r] \ar[r] &\widetilde{\Omega}_{R} \ar[r] \ar[r] \ar[r] &0,\\
0 \ar[r] \ar[r] \ar[r] &\underline{\mathcal{O}_{E_k}}  \ar[r] \ar[r] \ar[r] &\widetilde{\Pi}_{\mathrm{int},R} \ar[r] \ar[r] \ar[r] &\widetilde{\Pi}_{\mathrm{int},R} \ar[r] \ar[r] \ar[r] &0,\\
0 \ar[r] \ar[r] \ar[r] &\underline{{E_k}}  \ar[r] \ar[r] \ar[r] &\widetilde{\Pi}_{R} \ar[r] \ar[r] \ar[r] &\widetilde{\Pi}_{R} \ar[r] \ar[r] \ar[r] &0.\\
}
\]	
Similar statement could be mede to the corresponding fpd objects.	
\end{proposition}

\begin{proof}
See \cite[Theorem 4.5.7]{KL2}.	
\end{proof}

\

\section{Comparison in the Non-\'Etale Setting}

\indent We now consider the corresponding analog of the corresponding \cite[Chapter 4.6]{KL2} over $\mathrm{Spa}(R,R^+)$ as above which is just assumed to be analytic.


\begin{remark}
As we mentioned before, the corresponding analyticity of the ring $R$ will still not deduce from the analyticity of the corresponding rational Robba rings and sheaves.	
\end{remark}

\begin{theorem}\mbox{\bf{(After Kedlaya-Liu \cite[Theorem 4.6.1]{KL2})}}
Assume the corresponding Robba rings with respect to closed intervals are sheafy. Consider the corresponding categories in the following:\\
A. The corresponding category of all the corresponding \'etale-stably-pseudocoherent sheaves over the adic Fargues-Fontaine curve (associated to $\{\widetilde{\Pi}_{I,R,A}\}_{\{I\subset (0,\infty)\}}$) in the corresponding \'etale topology;\\
B. The corresponding category of all the corresponding \'etale-stably-pseudocoherent $F^k$-sheaves over families of Robba rings $\{\widetilde{\Pi}_{I,R,A}\}_{\{I\subset (0,\infty)\}}$;\\
C. The corresponding category of all the corresponding \'etale-stably-pseudocoherent modules over any Robba rings $\widetilde{\Pi}_{I,R,A}$ such that the interval $I=[r_1,r_2]$ satisfies that $0\leq r_1 \leq r_2/p^{hk}$;\\
D. The corresponding category of all the corresponding pseudocoherent modules over any Robba rings $\widetilde{\Pi}_{R,A}$, but admitting models of strictly-pseudocoherent $F^k$-modules over some $\widetilde{\Pi}_{r,R,A}$ for some radius $r>0$ whose base changes to some $\widetilde{\Pi}_{[r_1,r_2],R,A}$ will produce corresponding $F^k$-modules which are \'etale-stably-pseudocoherent;\\
E. The corresponding category of all the corresponding strictly-pseudocoherent $F^k$-modules over some $\widetilde{\Pi}_{\infty,R,A}$ whose base changes to some $\widetilde{\Pi}_{[r_1,r_2],R,A}$ will provide corresponding $F^k$-modules which are \'etale-stably-pseudocoherent;\\
F. The corresponding category of all the corresponding \'etale-stably-pseudocoherent sheaves over the adic Fargues-Fontaine curve (associated to $\{\widetilde{\Pi}_{I,R,A}\}_{\{I\subset (0,\infty)\}}$) in the corresponding pro-\'etale topology.\\
Then we have that they are actually equivalent.

\end{theorem}

\begin{proof}
This is just parallel \cite[Theorem 4.6.1]{KL2}, see \cite[Theorem 4.11]{XT2}.	
\end{proof}

\begin{theorem}\mbox{\bf{(After Kedlaya-Liu \cite[Theorem 4.6.1]{KL2})}}
Assume $A$ is sousperfectoid. Consider the corresponding categories of sheaves over $\mathrm{Spa}(R,R^+)_{\text{p\'et}}$ in the following:\\
A. The corresponding category of all the corresponding \'etale-stably-pseudocoherent modules over any Robba rings $\widetilde{\Pi}_{I,*,A}$ such that the interval $I=[r_1,r_2]$ satisfies that $0\leq r_1 \leq r_2/p^{hk}$;\\
B. The corresponding category of all the corresponding pseudocoherent modules over any Robba rings $\widetilde{\Pi}_{*,A}$, but admit models of strictly-pseudocoherent $F^k$-modules over some $\widetilde{\Pi}_{r,*,A}$ for some radius $r>0$ whose base changes to some $\widetilde{\Pi}_{[r_1,r_2],R,A}$ will provide corresponding $F^k$-modules which are \'etale-stably-pseudocoherent;\\
C. The corresponding category of all the corresponding \'etale-strictly-pseudocoherent $F^k$-modules over some $\widetilde{\Pi}_{\infty,*,A}$ whose base changes to some $\widetilde{\Pi}_{[r_1,r_2],*,A}$ will provide corresponding $F^k$-modules which are \'etale-stably-pseudocoherent.\\
Then we have that they are actually equivalent.

\end{theorem}

\begin{proof}
This is parallel to \cite[Theorem 4.6.1]{KL2}The proof is the same as the one where $*$ is just a ring $R$. See \cite[Theorem 4.11]{XT2}.	
\end{proof}

\begin{theorem}\mbox{\bf{(After Kedlaya-Liu \cite[Corollary 4.6.2]{KL2})}}
The corresponding analogs of the previous two theorems hold in the corresponding finite projective setting. Here we assume that $A$ is in the same hypotheses as above.	
\end{theorem}

\begin{proof}
See \cite[Corollary 4.6.2]{KL2}.	
\end{proof}


\indent Now we drop the corresponding condition on the sheafiness on $A$ by using the corresponding derived spectrum $\mathrm{Spa}^h(A)$ from \cite{BK1} where we have the corresponding $\infty$-sheaf $\mathcal{O}_{\mathrm{Spa}^h(A)}$, we now apply the corresponding construction to the Robba rings (in the rational setting over $\mathbb{Q}_p$) $*_{R,A}$ and denote the corresponding derived version by $*^h_{R,A}$. Here $A$ is assumed to be just Banach over $\mathbb{Q}_p$ or $\mathbb{F}_p((t))$, which is certainly in a very general situation.

\begin{remark}
One can apply the corresponding results of \cite{CS} to achieve so as well, which we believe will be robust as well.	
\end{remark}


\begin{theorem}\mbox{\bf{(After Kedlaya-Liu \cite[Theorem 4.6.1]{KL2})}}
Take a derived rational localization of $\widetilde{\Pi}_{I_1\bigcup I_2,R,A}$, which we denote it by $\widetilde{\Pi}^{h,*}_{I_1\bigcup I_2,R,A}$, where we choose two overlapped closed intervals $I_1$ and $I_2$. Now take the corresponding base changes of this localization along:
\begin{align}
\widetilde{\Pi}_{I_1\bigcup I_2,R,A}\rightarrow \widetilde{\Pi}_{I_1,R,A},\\
\widetilde{\Pi}_{I_1\bigcup I_2,R,A}\rightarrow \widetilde{\Pi}_{I_2,R,A},
\end{align}
which we will denote by $\widetilde{\Pi}^{h,*}_{I_1,R,A}$ and $\widetilde{\Pi}^{h,*}_{I_2,R,A}$. Consider the corresponding categories in the following:\\
A. The corresponding product of the category of all the corresponding finite projective $F^k$-modules over the Robba ring $\widetilde{\Pi}^{h,*}_{I_1,R,A}$ and the the category of all the corresponding finite projective $F^k$-modules over the Robba ring $\widetilde{\Pi}^{h,*}_{I_2,R,A}$, over the category of all the corresponding finite projective $F^k$-modules over $\widetilde{\Pi}^{h,*}_{I_1,R,A}\widehat{\otimes}^\mathbb{L}\widetilde{\Pi}^{h,*}_{I_2,R,A}$;\\
B. The corresponding category of all the corresponding finite projective modules over the Robba ring $\widetilde{\Pi}^{h,*}_{I_1\bigcup I_2,R,A}$.\\
Then we have that they are actually equivalent.

\end{theorem}

\begin{proof}
We consider a pair of overlapped intervals $I_1=[r_1,r_2],I_2=[s_1,s_2]$, and we have the corresponding $\infty$-Robba rings which form the desired situation, where we have the sequence which is basically exact up to higher homotopy:
\[
\xymatrix@C+2pc@R+3pc{
\widetilde{\Pi}^{h,*}_{I_1\bigcup I_2,R,A} \ar[r]\ar[r] \ar[r] \ar[r]  &\widetilde{\Pi}^{h,*}_{I_1,R,A}\bigoplus \widetilde{\Pi}^{h,*}_{I_2,R,A} \ar[r] \ar[r] \ar[r]\ar[r] &\widetilde{\Pi}^{h,*}_{I_1,R,A}\widehat{\otimes}^\mathbb{L}\widetilde{\Pi}^{h,*}_{I_2,R,A}.  
}
\] 
Now we consider the corresponding two finite projective module spectra (which admit retracts from finite free module spectra as in \cite[Proposition 7.2.2.7]{Lu1}) $M_1,M_2,M_{12}$ over the rings in the middle and the rightmost positions, and assume that they form a corresponding glueing datum. Then we consider the corresponding presentation diagram for the module spectra:
\[
\xymatrix@C+2pc@R+3pc{
 &F \ar[r] \ar[r] \ar[r] \ar[d] \ar[d] \ar[d] &F_1\bigoplus F_2 \ar[d] \ar[d] \ar[d] \ar[r] \ar[r] \ar[r] &F_{12} \ar[d] \ar[d] \ar[d] \\
&G \ar[r] \ar[r] \ar[r] \ar[d] \ar[d] \ar[d] &G_1\bigoplus G_2 \ar[d]^{r_1\oplus r_2} \ar[d] \ar[d] \ar[r] \ar[r] \ar[r] &G_{12} \ar[d]^{r_{12}} \ar[d] \ar[d] \\
&M \ar[r] \ar[r] \ar[r] &M_1\bigoplus M_2 \ar[r] \ar[r] \ar[r] &M_{12} \\
}
\]  
by just taking the corresponding fibers, here $G,G_1,G_2,G_{12}$ are the finite free modules and we have the corresponding retracts $r_1,r_2,r_{12}$, but note that the corresponding map $r_1\oplus r_2$ might not be a priori a retract. However consider the following commutative diagram:
\[
\xymatrix@C+2pc@R+3pc{
& &0 \ar[d] \ar[d] \ar[d] &0 \ar[d] \ar[d] \ar[d] \\
 &\pi_0 F \ar[r] \ar[r] \ar[r] \ar[d] \ar[d] \ar[d] &\pi_0F_1\bigoplus \pi_0F_2 \ar[d] \ar[d] \ar[d] \ar[r] \ar[r] \ar[r] &\pi_0F_{12} \ar[d] \ar[d] \ar[d] \\
&\pi_0G \ar[r] \ar[r] \ar[r] \ar[d] \ar[d] \ar[d] &\pi_0G_1\bigoplus \pi_0G_2 \ar[d]^{\pi_0(r_1)\oplus \pi_0(r_2)} \ar[d] \ar[d] \ar[r] \ar[r] \ar[r] &\pi_0G_{12} \ar[d]^{r_{12}} \ar[d] \ar[d] \ar[r] \ar[r] \ar[r]&0 \\
&\pi_0M \ar[r] \ar[r] \ar[r] &\pi_0M_1\bigoplus \pi_0M_2\ar[d] \ar[d] \ar[d]  \ar[r] \ar[r] \ar[r] &\pi_0M_{12}\ar[d] \ar[d] \ar[d]  \ar[r] \ar[r] \ar[r]&0 \\
&&0&0.
}
\] 
We then consider the corresponding the argument of Kedlaya as in \cite[Proposition 5.11]{XT3} which gives us modification on sections $\pi_0s_1$ and $\pi_0 s_2$ of $\pi_0 r_1$ and $\pi_0 r_2$ respectively such that we have (with the same notations) the new $\pi_0 s_1$ and $\pi_0 s_2$ give the corresponding new lifts $s_1$ and $s_2$ whose coproduct actually restricts to $M$ which produce a section for $M$ for the map $G\rightarrow M$, which proves the desired finite projectivity on $M$. From $A$ to $B$ we just take the corresponding projection while going back we consider the corresponding binary glueing to tackle the corresponding interval which cannot be reached by taking Frobenius translations but could be covered by two ones reachable by taking Frobenius translations. The lifts could be also modified by directly on the derived level.\\
\end{proof}

\indent In \cite{GV}, Galatius and Venkatesh considered some derived Galois deformation theory. Now we make some discussion around this point by using the corresponding simplicial Banach rings (over $\mathbb{Q}_p$ or over $\mathbb{F}_p((t))$) in \cite{BK1}, namely we use the notation $A^h$ to denote any local charts \footnote{Even one can take the more general simplicial Banach rings as in \cite{BBBK}, especially one would like to focus on the corresponding analytification of derived Galois deformation rings as in \cite{GV}, at least our feeling is that \cite{BBBK} will allow one to take the corresponding derived adic generic fiber to produce some desired simplicial Banach rings to tackle derived Galois deformation problems in \cite{GV}.} coming from the corresponding Bambozzi-Kremnizer spectrum in \cite{BK1} attached to some Banach algebra over (over $\mathbb{Q}_p$ or over $\mathbb{F}_p((t))$). Over the analytic field specified we take the corresponding completed tensor products the period rings involved with $A^h$, which we will denote by $*_{R,A^h}$. Then we have the following similar results as above:

\begin{conjecture}\mbox{\bf{(After Kedlaya-Liu \cite[Theorem 4.6.1]{KL2})}}
Consider the corresponding categories in the following:\\
A. The corresponding category of all the corresponding locally finite free sheaves over the adic Fargues-Fontaine curve (associated to $\{\widetilde{\Pi}_{I,R,A^h}\}_{\{I\subset (0,\infty)\}}$) in the corresponding homotopy Zariski topology from \cite{BK1} and \cite{BBBK} induced by Koszul derived rational localization;\\
B. The corresponding category of all the corresponding finite projective $F^k$-sheaves over families of Robba rings $\{\widetilde{\Pi}_{I,R,A^h}\}_{\{I\subset (0,\infty)\}}$;\\
C. The corresponding category of all the corresponding finite projective modules over any Robba rings $\widetilde{\Pi}_{I,R,A^h}$ such that the interval $I=[r_1,r_2]$ satisfies that $0\leq r_1 \leq r_2/p^{hk}$.\\
Then we have that they are actually equivalent.

\end{conjecture}

\begin{theorem}\mbox{\bf{(After Kedlaya-Liu \cite[Theorem 4.6.1]{KL2})}}
Consider the corresponding categories in the following:\\
A. The corresponding category of all the corresponding finite projective $F^k$-sheaves over families of Robba rings $\{\widetilde{\Pi}_{I,R,A^h}\}_{\{I\subset (0,\infty)\}}$;\\
B. The corresponding category of all the corresponding finite projective modules over any Robba rings $\widetilde{\Pi}_{I,R,A^h}$ such that the interval $I=[r_1,r_2]$ satisfies that $0\leq r_1 \leq r_2/p^{hk}$.\\
Then we have that they are actually equivalent.

\end{theorem}

\begin{proof}
Note that in our situation we do have the following nice short exact sequences:
\[
\xymatrix@C+0pc@R+0pc{
0\ar[r] \ar[r] \ar[r] &\pi_k \widetilde{\Pi}_{I_1\bigcup I_2,R,A^h} \ar[r] \ar[r] \ar[r] &\pi_k\widetilde{\Pi}_{I_1,R,A^h}\bigoplus \pi_k\widetilde{\Pi}_{I_2,R,A^h} \ar[r] \ar[r] \ar[r] &\pi_k\widetilde{\Pi}_{I_1\bigcap I_2,R,A^h} \ar[r] \ar[r] \ar[r] &0, k=0,1,....
}
\]
We consider a pair of overlapped intervals $I_1=[r_1,r_2],I_2=[s_1,s_2]$, and we have the corresponding $\infty$-Robba rings which form a corresponding glueing sequence in the sense \cite[Definition 2.7.3]{KL1}:
\[
\xymatrix@C+0pc@R+0pc{
0\ar[r] \ar[r] \ar[r] &\pi_0\widetilde{\Pi}_{I_1\bigcup I_2,R,A^h} \ar[r] \ar[r] \ar[r] &\pi_0\widetilde{\Pi}_{I_1,R,A^h}\bigoplus \pi_0\widetilde{\Pi}_{I_2,R,A^h} \ar[r] \ar[r] \ar[r] &\pi_0\widetilde{\Pi}_{I_1\bigcap I_2,R,A^h} \ar[r] \ar[r] \ar[r] &0.
}
\] 	

Now we consider the corresponding two finite projective module spectra (which admit retracts from finite free module spectra as in \cite[Proposition 7.2.2.7]{Lu1}) $M_1,M_2,M_{12}$ over the rings in the middle and the rightmost position, and assume that they form a corresponding glueing datum. Then we consider the corresponding presentation diagram for the module spectra:
\[
\xymatrix@C+2pc@R+3pc{
 &F \ar[r] \ar[r] \ar[r] \ar[d] \ar[d] \ar[d] &F_1\bigoplus F_2 \ar[d] \ar[d] \ar[d] \ar[r] \ar[r] \ar[r] &F_{12} \ar[d] \ar[d] \ar[d] \\
&G \ar[r] \ar[r] \ar[r] \ar[d] \ar[d] \ar[d] &G_1\bigoplus G_2 \ar[d]^{r_1\oplus r_2} \ar[d] \ar[d] \ar[r] \ar[r] \ar[r] &G_{12} \ar[d]^{r_{12}} \ar[d] \ar[d] \\
&M \ar[r] \ar[r] \ar[r] &M_1\bigoplus M_2 \ar[r] \ar[r] \ar[r] &M_{12} \\
}
\]  
by just taking the corresponding fibers, here $G,G_1,G_2,G_{12}$ are the finite free modules and we have the corresponding retracts $r_1,r_2,r_{12}$, but note that the corresponding map $r_1\oplus r_2$ might not be a priori a retract. However consider the following commutative diagram (which admits more exactness than in the previous theorem):
\[
\xymatrix@C+2pc@R+3pc{
&& &0 \ar[d] \ar[d] \ar[d] &0 \ar[d] \ar[d] \ar[d] \\
& &\pi_0 F \ar[r] \ar[r] \ar[r] \ar[d] \ar[d] \ar[d] &\pi_0F_1\bigoplus \pi_0F_2 \ar[d] \ar[d] \ar[d] \ar[r] \ar[r] \ar[r] &\pi_0F_{12} \ar[d] \ar[d] \ar[d] \\
&0 \ar[r] \ar[r] \ar[r] &\pi_0G \ar[r] \ar[r] \ar[r] \ar[d] \ar[d] \ar[d] &\pi_0G_1\bigoplus \pi_0G_2 \ar[d]^{\pi_0(r_1)\oplus \pi_0(r_2)} \ar[d] \ar[d] \ar[r] \ar[r] \ar[r] &\pi_0G_{12} \ar[d]^{r_{12}} \ar[d] \ar[d] \ar[r] \ar[r] \ar[r]&0 \\
&0 \ar[r] \ar[r] \ar[r] &\pi_0M \ar[r] \ar[r] \ar[r] &\pi_0M_1\bigoplus \pi_0M_2\ar[d] \ar[d] \ar[d]  \ar[r] \ar[r] \ar[r] &\pi_0M_{12}\ar[d] \ar[d] \ar[d]  \ar[r] \ar[r] \ar[r]&0 \\
&&&0&0.
}
\] 
We then consider the corresponding the argument of Kedlaya as in \cite[Proposition 5.11]{XT3} which gives us modification on sections $\pi_0s_1$ and $\pi_0 s_2$ of $\pi_0 r_1$ and $\pi_0 r_2$ respectively such that we have (with the same notations) the new $\pi_0 s_1$ and $\pi_0 s_2$ give the corresponding new lifts $s_1$ and $s_2$ whose coproduct actually restricts to $M$ which produce a section for $M$ for the map $G\rightarrow M$, which proves the desired finite projectivity on $M$. From $A$ to $B$ we just take the corresponding projection while going back we consider the corresponding binary glueing to tackle the corresponding interval which cannot be reached by taking Frobenius translations but could be covered by two ones reachable by taking Frobenius translations. The lifts could be also modified by directly on the derived level. Or one can just show the flatness on the derived level directly by using derived Tor, which will directly proves that $M$ is finite projective as a module spectrum.
\end{proof}

\indent Now we even drop the hypothesis (by using the corresponding techniques in \cite{XT3} and \cite{XT4}) on the commutativity on $A$, namely now $A$ could be allowed to be strictly quotient coming from the free Tate algebras $\mathbb{Q}_p\left<Z_1,...,Z_d\right>$ and $\mathbb{F}_p((t))\left<Z_1,...,Z_d\right>$, which we will use some different notation $B$ to denote this. Note certainly that $A$ is some special case of such $B$, therefore now we are going to indeed generalize the discussion above to the noncommutative setting.

\begin{remark}
One can definitely consider more general coefficients in the Banach setting, but for Hodge-Iwasawa theory, we prefer to locate our discussion in the area where we could get the interesting rings by taking admissible and strictly quotient from free Tate algebras over analytic field, which certainly includes the corresponding situation of rigid analytic affinoids.	
\end{remark}

\begin{setting}
By the work \cite{XT3} and \cite{XT4} we can freely translate between the language of $B$-stably-pseudocoherent sheaves over analytic or \'etale topology and \'etale setting and the corresponding $B$-stably pseudocoherent modules or $B$-\'etale stably pseudocoherent modules, which makes the discussion in our context possible. What was happening indeed when we are in the corresponding sheafy tensor product situation is that one can direct read off such result in the next theorem by using the descent for the pseudocoherent modules with desired stability in \cite{KL2} and \cite{Ked1}.  	
\end{setting}

\begin{theorem}\mbox{\bf{(After Kedlaya-Liu \cite[Theorem 4.6.1]{KL2})}}
Assume the corresponding Robba rings with respect to closed intervals are sheafy. Consider the corresponding categories in the following (all modules are left):\\
A. The corresponding category of all the corresponding $B$-\'etale-stably-pseudocoherent sheaves over the adic Fargues-Fontaine curve (associated to $\{\widetilde{\Pi}_{I,R,B}\}_{\{I\subset (0,\infty)\}}$) in the corresponding \'etale topology;\\
B. The corresponding category of all the corresponding $B$-\'etale-stably-pseudocoherent $F^k$-sheaves over families of Robba rings $\{\widetilde{\Pi}_{I,R,B}\}_{\{I\subset (0,\infty)\}}$;\\
C. The corresponding category of all the corresponding $B$-\'etale-stably-pseudocoherent modules over any Robba rings $\widetilde{\Pi}_{I,R,B}$ such that the interval $I=[r_1,r_2]$ satisfies that $0\leq r_1 \leq r_2/p^{hk}$;\\
D. The corresponding category of all the corresponding $B$-pseudocoherent modules over any Robba rings $\widetilde{\Pi}_{R,B}$, but admit models of $B$-strictly-pseudocoherent $F^k$-modules over some $\widetilde{\Pi}_{r,R,B}$ for some radius $r>0$ whose base changes to some $\widetilde{\Pi}_{[r_1,r_2],R,B}$ will provide corresponding $F^k$-modules which are $B$-\'etale-stably-pseudocoherent;\\
E. The corresponding category of all the corresponding $B$-strictly-pseudocoherent $F^k$-modules over some $\widetilde{\Pi}_{\infty,R,B}$ whose base changes to some $\widetilde{\Pi}_{[r_1,r_2],R,B}$ will provide corresponding $F^k$-modules which are $B$-\'etale-stably-pseudocoherent.\\
Then we have that they are actually equivalent.

\end{theorem}


\begin{proof}
This is just parallel \cite[Theorem 4.6.1]{KL2}, see \cite[Theorem 4.11]{XT2}. The corresponding modules could be regarded as sheaves over the deformed sites as in our previous work \cite{XT3} and \cite{XT4}.	
\end{proof}

\begin{theorem}\mbox{\bf{(After Kedlaya-Liu \cite[Theorem 4.6.1]{KL2})}}
Consider the corresponding categories of sheaves over $\mathrm{Spa}(R,R^+)_{\text{\'et}}$ in the following (all the modules are left over the corresponding rings):\\
A. The corresponding category of all the corresponding $B$-\'etale-stably-pseudocoherent modules over any Robba rings $\widetilde{\Pi}_{I,*,B}$ such that the interval $I=[r_1,r_2]$ satisfies that $0\leq r_1 \leq r_2/p^{hk}$;\\
B. The corresponding category of all the corresponding $B$-pseudocoherent modules over any Robba rings $\widetilde{\Pi}_{*,B}$, but admit models of $B$-strictly-pseudocoherent $F^k$-modules over some $\widetilde{\Pi}_{r,*,B}$ for some radius $r>0$ whose base changes to some $\widetilde{\Pi}_{[r_1,r_2],R,B}$ will provide corresponding $F^k$-modules which are $B$-\'etale-stably-pseudocoherent;\\
C. The corresponding category of all the corresponding $B$-strictly-pseudocoherent $F^k$-modules over some $\widetilde{\Pi}_{\infty,*,B}$ whose base changes to some $\widetilde{\Pi}_{[r_1,r_2],*,B}$ will provide corresponding $F^k$-modules which are $B$-\'etale-stably-pseudocoherent.\\
Then we have that they are actually equivalent.

\end{theorem}

\begin{proof}
This is parallel to \cite[Theorem 4.6.1]{KL2}The proof is the same as the one where $*$ is just a ring $R$. See \cite[Theorem 4.11]{XT2}.	
\end{proof}

\begin{theorem}\mbox{\bf{(After Kedlaya-Liu \cite[Corollary 4.6.2]{KL2})}}
The corresponding analogs of the previous two theorems hold in the corresponding finite projective setting (one has to consider the bimodules). Here we assume that $B$ is in the same hypotheses as above.	
\end{theorem}

\begin{proof}
See \cite[Corollary 4.6.2]{KL2} and \cite[Proposition 5.12]{XT3}.	
\end{proof}

\

\section{Discussion on the Generality of Gabber-Ramero}

\subsection{General Period Rings}

\indent We can encode now the corresponding discussion in the previous section actually, but we choose to separately discuss the corresponding results in detail here. Certainly many results in \cite{KL1} and \cite{KL2}, and \cite{XT1}, \cite{XT2}, \cite{XT3} and \cite{XT4} rely on the corresponding topologically nilpotent units and systems of topologically nilpotents. Therefore we will discuss the corresponding admissible and reasonable generalization after \cite{KL1}, \cite{KL2}, \cite{XT1}, \cite{XT2}, \cite{XT3}, \cite{XT4} and \cite{GR}.

\begin{setting}
Drop the condition on the analyticity on $(R,R^+)$ and we switch to the corresponding notation $(S,S^+)$.	
\end{setting}


After Kedlaya-Liu \cite[Definition 4.1.1]{KL2}, we consider the following constructions. First we consider the corresponding Witt vectors coming from the corresponding adic ring $(S,S^+)$. First we consider the corresponding generalized Witt vectors with respect to $(S,S^+)$ with the corresponding coefficients in the Tate algebra with the general notation $W(S^+)[[S]]$. The general form of any element in such deformed ring could be written as $\sum_{i\geq 0,i_1\geq 0,...,i_d\geq 0}\pi^i[\overline{y}_i]X_1^{i_1}...X_d^{i_d}$. Then we take the corresponding completion with respect to the following norm for some radius $t>0$:
\begin{align}
\|.\|_{t,A}(\sum_{i\geq 0,i_1\geq 0,...,i_d\geq 0}\pi^i[\overline{y}_i]X_1^{i_1}...X_d^{i_d}):= \max_{i\geq 0,i_1\geq 0,...,i_d\geq 0}p^{-i}\|.\|_S(\overline{y}_i)	
\end{align}
which will give us the corresponding ring $\widetilde{\Pi}_{\mathrm{int},t,S,A}$ such that we could put furthermore that:
\begin{align}
\widetilde{\Pi}_{\mathrm{int},S,A}:=\bigcup_{t>0} \widetilde{\Pi}_{\mathrm{int},t,S,A}.	
\end{align}
Then as in \cite[Definition 4.1.1]{KL2}, we now put the ring $\widetilde{\Pi}_{\mathrm{bd},t,S,A}:=\widetilde{\Pi}_{\mathrm{int},t,S,A}[1/\pi]$ and we set:
\begin{align}
\widetilde{\Pi}_{\mathrm{bd},S,A}:=\bigcup_{t>0} \widetilde{\Pi}_{\mathrm{bd},t,S,A}.	
\end{align}
The corresponding Robba rings with respect to some intervals and some radius could be defined in the same way as in \cite[Definition 4.1.1]{KL2}. To be more precise we consider the completion of the corresponding ring $W(S^+)[[S]][1/\pi]$ with respect to the following norm for some $t>0$ where $t$ lives in some prescribed interval $I=[s,r]$: 
\begin{align}
\|.\|_{t,A}(\sum_{i,i_1\geq 0,...,i_d\geq 0}\pi^i[\overline{y}_i]X_1^{i_1}...X_d^{i_d}):= \max_{i\geq 0,i_1\geq 0,...,i_d\geq 0}p^{-i}\|.\|_S(\overline{y}_i).	
\end{align}
This process will produce the corresponding Robba rings with respect to  the given interval $I=[s,r]$. Now for particular sorts of intervals $(0,r]$ we will have the corresponding Robba ring $\widetilde{\Pi}_{r,S,A}$ and we will have the corresponding Robba ring $\widetilde{\Pi}_{\infty,S,A}$	if the corresponding interval is taken to be $(0,\infty)$. Then in our situation we could just take the corresponding union throughout all the radius $r>0$ to define the corresponding full Robba ring taking the notation of $\widetilde{\Pi}_{S,A}$.

\indent The corresponding Robba rings $\widetilde{\Pi}_{\mathrm{bd},S,A}$, $\widetilde{\Pi}_{S,A}$, $\widetilde{\Pi}_{I,S,A}$, $\widetilde{\Pi}_{r,S,A}$, $\widetilde{\Pi}_{\infty,S,A}$ are actually themselves Tate adic Banach rings. However in many further application the non-Tateness of the ring $S$ will cause some reason for us to do the corresponding modification.

\indent Then for any general affinoid algebra $A$ over the corresponding base analytic field, we just take the corresponding quotients of the corresponding rings defined in the previous definition over some Tate algebras in rigid analytic geometry, with the same notations though $A$ now is more general. Note that one can actually show that the definition does not depend on the corresponding choice of the corresponding presentations over $A$.

\indent Again in this situation more generally, the corresponding Robba rings $\widetilde{\Pi}_{\mathrm{bd},S,A}$, $\widetilde{\Pi}_{S,A}$, $\widetilde{\Pi}_{I,S,A}$, $\widetilde{\Pi}_{r,S,A}$, $\widetilde{\Pi}_{\infty,S,A}$ are actually themselves Tate adic Banach rings.

\begin{lemma} \mbox{\bf{(After Kedlaya-Liu \cite[Lemma 5.2.6]{KL2})}}
For any two radii $0<r_1<r_2$ we have the corresponding equality:
\begin{align}
\widetilde{\Pi}_{\mathrm{int},r_2,S,\mathbb{Q}_p\{T_1,...,T_d\}}\bigcap \widetilde{\Pi}_{[r_1,r_2],S,\mathbb{Q}_p\{T_1,...,T_d\}}	=\widetilde{\Pi}_{\mathrm{int},r_1,S,\mathbb{Q}_p\{T_1,...,T_d\}}.
\end{align}

\end{lemma}

\begin{proof}
See \cite[Lemma 5.2.6]{KL2} and \cite[Proposition 2.13]{XT2}.	
\end{proof}

\begin{lemma} \mbox{\bf{(After Kedlaya-Liu \cite[Lemma 5.2.6]{KL2})}}
For any two radii $0<r_1<r_2$ we have the corresponding equality:
\begin{align}
\widetilde{\Pi}_{\mathrm{int},r_2,S,\mathbb{F}_p((t))\{T_1,...,T_d\}}\bigcap \widetilde{\Pi}_{[r_1,r_2],S,\mathbb{F}_p((t))\{T_1,...,T_d\}}	=\widetilde{\Pi}_{\mathrm{int},r_1,S,\mathbb{F}_p((t))\{T_1,...,T_d\}}.
\end{align}

\end{lemma}

\begin{proof}
See \cite[Lemma 5.2.6]{KL2} and \cite[Proposition 2.13]{XT2}.	
\end{proof}

\begin{lemma} \mbox{\bf{(After Kedlaya-Liu \cite[Lemma 5.2.6]{KL2})}}
For general affinoid $A$ as above (over $\mathbb{Q}_p$ or $\mathbb{F}_p((t))$) and for any two radii $0<r_1<r_2$ we have the corresponding equality:
\begin{align}
\widetilde{\Pi}_{\mathrm{int},r_2,S,A}\bigcap \widetilde{\Pi}_{[r_1,r_2],S,A}	=\widetilde{\Pi}_{\mathrm{int},r_1,S,A}.
\end{align}

\end{lemma}

\begin{proof}
See \cite[Lemma 5.2.6]{KL2} and \cite[Proposition 2.14]{XT2}.	
\end{proof}

\begin{lemma} \mbox{\bf{(After Kedlaya-Liu \cite[Lemma 5.2.10]{KL2})}}
For any four radii $0<r_1<r_2<r_3<r_4$ we have the corresponding equality:
\begin{align}
\widetilde{\Pi}_{[r_1,r_3],S,\mathbb{Q}_p\{T_1,...,T_d\}}\bigcap \widetilde{\Pi}_{[r_2,r_4],S,\mathbb{Q}_p\{T_1,...,T_d\}}	=\widetilde{\Pi}_{[r_1,r_4],S,\mathbb{Q}_p\{T_1,...,T_d\}}.
\end{align}

\end{lemma}

\begin{proof}
See \cite[Lemma 5.2.10]{KL2} and \cite[Proposition 2.16]{XT2}.	
\end{proof}

\begin{lemma} \mbox{\bf{(After Kedlaya-Liu \cite[Lemma 5.2.10]{KL2})}}
For any four radii $0<r_1<r_2<r_3<r_4$ we have the corresponding equality:
\begin{align}
\widetilde{\Pi}_{[r_1,r_3],S,\mathbb{F}_p((t))\{T_1,...,T_d\}}\bigcap \widetilde{\Pi}_{[r_2,r_4],S,\mathbb{F}_p((t))\{T_1,...,T_d\}}	=\widetilde{\Pi}_{[r_1,r_4],S,\mathbb{F}_p((t))\{T_1,...,T_d\}}.
\end{align}

\end{lemma}

\begin{proof}
See \cite[Lemma 5.2.10]{KL2} and \cite[Proposition 2.16]{XT2}.	
\end{proof}

\begin{lemma} \mbox{\bf{(After Kedlaya-Liu \cite[Lemma 5.2.10]{KL2})}}
For any four radii $0<r_1<r_2<r_3<r_4$ we have the corresponding equality:
\begin{align}
\widetilde{\Pi}_{[r_1,r_3],S,A}\bigcap \widetilde{\Pi}_{[r_2,r_4],S,A}	=\widetilde{\Pi}_{[r_1,r_4],S,A}.
\end{align}

\end{lemma}

\begin{proof}
See \cite[Lemma 5.2.10]{KL2} and \cite[Proposition 2.17]{XT2}.	
\end{proof}

\begin{definition}
All the Frobenius finite projective, pseudocoherent and fpd modules under $F^k$ could be defined in the exact same way as in \cref{chapter3}. We will not repeat the corresponding definition, but note that we change the notation for $R$ to be just $S$.	
\end{definition}

\begin{conjecture}\mbox{\bf{(After Kedlaya-Liu \cite[Theorem 4.6.1]{KL2})}}
Consider the corresponding categories in the following:\\
A. The corresponding category of all the corresponding locally finite free sheaves over the adic Fargues-Fontaine curve (associated to $\{\widetilde{\Pi}_{I,S,A^h}\}_{\{I\subset (0,\infty)\}}$) in the corresponding homotopy Zariski topology from \cite{BK1} and \cite{BBBK} induced by Koszul derived rational localization;\\
B. The corresponding category of all the corresponding finite projective $F^k$-sheaves over families of Robba rings $\{\widetilde{\Pi}_{I,S,A^h}\}_{\{I\subset (0,\infty)\}}$;\\
C. The corresponding category of all the corresponding finite projective modules over any Robba rings $\widetilde{\Pi}_{I,S,A^h}$ such that the interval $I=[r_1,r_2]$ satisfies that $0\leq r_1 \leq r_2/p^{hk}$.\\
Then we have that they are actually equivalent.

\end{conjecture}

\begin{theorem}\mbox{\bf{(After Kedlaya-Liu \cite[Theorem 4.6.1]{KL2})}}
Consider the corresponding categories in the following:\\
A. The corresponding category of all the corresponding finite projective $F^k$-sheaves over families of Robba rings $\{\widetilde{\Pi}_{I,S,A^h}\}_{\{I\subset (0,\infty)\}}$;\\
B. The corresponding category of all the corresponding finite projective modules over any Robba rings $\widetilde{\Pi}_{I,S,A^h}$ such that the interval $I=[r_1,r_2]$ satisfies that $0\leq r_1 \leq r_2/p^{hk}$.\\
Then we have that they are actually equivalent.

\end{theorem}

\begin{proof}
As in the situation in the analytic setting we consider the following argument which will show the desired result. In our situation we do have the following nice short exact sequences:
\[
\xymatrix@C+0pc@R+0pc{
0\ar[r] \ar[r] \ar[r] &\pi_k \widetilde{\Pi}_{I_1\bigcup I_2,S,A^h} \ar[r] \ar[r] \ar[r] &\pi_k\widetilde{\Pi}_{I_1,S,A^h}\bigoplus \pi_k\widetilde{\Pi}_{I_2,S,A^h} \ar[r] \ar[r] \ar[r] &\pi_k\widetilde{\Pi}_{I_1\bigcap I_2,S,A^h} \ar[r] \ar[r] \ar[r] &0, k=0,1,....
}
\]
We consider a pair of overlapped intervals $I_1=[r_1,r_2],I_2=[s_1,s_2]$, and we have the corresponding $\infty$-Robba rings which form a corresponding glueing sequence in the sense \cite[Definition 2.7.3]{KL1}:
\[
\xymatrix@C+0pc@R+0pc{
0\ar[r] \ar[r] \ar[r] &\pi_0\widetilde{\Pi}_{I_1\bigcup I_2,S,A^h} \ar[r] \ar[r] \ar[r] &\pi_0\widetilde{\Pi}_{I_1,S,A^h}\bigoplus \pi_0\widetilde{\Pi}_{I_2,S,A^h} \ar[r] \ar[r] \ar[r] &\pi_0\widetilde{\Pi}_{I_1\bigcap I_2,S,A^h} \ar[r] \ar[r] \ar[r] &0.
}
\]

Now we consider the corresponding two finite projective module spectra (which admits retracts from finite free module spectra as in \cite[Proposition 7.2.2.7]{Lu1}) $M_1,M_2,M_{12}$ over the rings in the middle and the rightmost position, and assume that they form a corresponding glueing datum. Then we consider the corresponding presentation diagram for the module spectra:
\[
\xymatrix@C+2pc@R+3pc{
 &L \ar[r] \ar[r] \ar[r] \ar[d] \ar[d] \ar[d] &L_1\bigoplus L_2 \ar[d] \ar[d] \ar[d] \ar[r] \ar[r] \ar[r] &L_{12} \ar[d] \ar[d] \ar[d] \\
&N \ar[r] \ar[r] \ar[r] \ar[d] \ar[d] \ar[d] &N_1\bigoplus N_2 \ar[d]^{r_1\oplus r_2} \ar[d] \ar[d] \ar[r] \ar[r] \ar[r] &N_{12} \ar[d]^{r_{12}} \ar[d] \ar[d] \\
&M \ar[r] \ar[r] \ar[r] &M_1\bigoplus M_2 \ar[r] \ar[r] \ar[r] &M_{12} \\
}
\]  
by just taking the corresponding fibers, here $N,N_1,N_2,N_{12}$ are the finite free modules and we have the corresponding retracts $r_1,r_2,r_{12}$, but note that the corresponding map $r_1\oplus r_2$ might not be a priori a retract. However consider the following commutative diagram:
\[
\xymatrix@C+2pc@R+3pc{
&& &0 \ar[d] \ar[d] \ar[d] &0 \ar[d] \ar[d] \ar[d] \\
& &\pi_0 F \ar[r] \ar[r] \ar[r] \ar[d] \ar[d] \ar[d] &\pi_0F_1\bigoplus \pi_0F_2 \ar[d] \ar[d] \ar[d] \ar[r] \ar[r] \ar[r] &\pi_0F_{12} \ar[d] \ar[d] \ar[d] \\
&0 \ar[r] \ar[r] \ar[r] &\pi_0G \ar[r] \ar[r] \ar[r] \ar[d] \ar[d] \ar[d] &\pi_0G_1\bigoplus \pi_0G_2 \ar[d]^{\pi_0(r_1)\oplus \pi_0(r_2)} \ar[d] \ar[d] \ar[r] \ar[r] \ar[r] &\pi_0G_{12} \ar[d]^{r_{12}} \ar[d] \ar[d] \ar[r] \ar[r] \ar[r]&0 \\
&0 \ar[r] \ar[r] \ar[r] &\pi_0M \ar[r] \ar[r] \ar[r] &\pi_0M_1\bigoplus \pi_0M_2\ar[d] \ar[d] \ar[d]  \ar[r] \ar[r] \ar[r] &\pi_0M_{12}\ar[d] \ar[d] \ar[d]  \ar[r] \ar[r] \ar[r]&0 \\
&&&0&0.
}
\] 
We then consider the corresponding the argument of Kedlaya as in \cite[Proposition 5.11]{XT3} which gives us modification on sections $\pi_0s_1$ and $\pi_0 s_2$ of $\pi_0 r_1$ and $\pi_0 r_2$ respectively such that we have (with the same notations) the new $\pi_0 s_1$ and $\pi_0 s_2$ give the coresponding new lifts $s_1$ and $s_2$ whose coproduct actually restricts to $M$ which produce a section for $M$ for the map $N\rightarrow M$, which proves the desired finite projectivity on $M$. From $A$ to $B$ we just take the corresponding projection while going back we consider the corresponding binary glueing to tackle the corresponding interval which cannot be reached by taking Frobenius translations but could be covered by two ones reachable by taking Frobenius translations.
One can actually derive this on the derived level directly by taking the corresponding lift of the difference of the base changes of $s_1$ and $s_2$ to the right in $\mathrm{Ext}^0(M_{12},L_{12})$ to $\mathrm{Ext}^0(M_1\bigoplus M_2,L_1\bigoplus L_2)$ to modify the corresponding sections $s_1$ and $s_2$ in order to restrict to $M$, which proves the result by the argument of Kedlaya as in \cite[Proposition 5.11]{XT3}.
\end{proof}


\indent Now again we even drop the hypothesis (by using the corresponding techniques in \cite{XT3} and \cite{XT4}) on the commutativity on $A$, namely now $A$ could be allowed to be strictly quotient coming from the free Tate algebras $\mathbb{Q}_p\left<Z_1,...,Z_d\right>$ and $\mathbb{F}_p((t))\left<Z_1,...,Z_d\right>$, which we will use some different notation $B$ to denote this. Note certainly that $A$ is some special case of such $B$, therefore now we are going to indeed generalize the discussion above to the noncommutative setting.

\begin{theorem}\mbox{\bf{(After Kedlaya-Liu \cite[Theorem 4.6.1]{KL2})}}
Assume the corresponding Robba rings with respect to closed intervals are sheafy. Consider the corresponding categories in the following (all modules are left):\\
A. The corresponding category of all the corresponding finite locally projective sheaves over the adic Fargues-Fontaine curve (associated to $\{\widetilde{\Pi}_{I,S,B}\}_{\{I\subset (0,\infty)\}}$) in the corresponding \'etale topology;\\
B. The corresponding category of all the corresponding finite projective $F^k$-sheaves over families of Robba rings $\{\widetilde{\Pi}_{I,S,B}\}_{\{I\subset (0,\infty)\}}$;\\
C. The corresponding category of all the corresponding finite projective modules over any Robba rings $\widetilde{\Pi}_{I,S,B}$ such that the interval $I=[r_1,r_2]$ satisfies that $0\leq r_1 \leq r_2/p^{hk}$.\\
Then we have that they are actually equivalent.

\end{theorem}

\begin{proof}
This is just parallel \cite[Theorem 4.6.1]{KL2}, see \cite[Theorem 4.11]{XT2}. The corresponding modules could be regarded as sheaves over the deformed sites as in our previous work \cite{XT3} and \cite{XT4}.	
\end{proof}

\begin{theorem}\mbox{\bf{(After Kedlaya-Liu \cite[Theorem 4.6.1]{KL2})}}
Consider the corresponding categories of sheaves over $\mathrm{Spa}(S,S^+)_{\text{\'et}}$ in the following (all the modules are bimodules over the corresponding rings):\\
A. The corresponding category of all the corresponding finite projective $F^k$-sheaves over families of Robba rings $\{\widetilde{\Pi}_{I,*,B}\}_{\{I\subset (0,\infty)\}}$;\\
B. The corresponding category of all the corresponding finite projective modules over any Robba rings $\widetilde{\Pi}_{I,*,B}$ such that the interval $I=[r_1,r_2]$ satisfies that $0\leq r_1 \leq r_2/p^{hk}$;\\
Then we have that they are actually equivalent.

\end{theorem}

\begin{proof}
Note that in this situation the corresponding space is not even analytic adic space, but the Robba rings over some certain perfectoid domains are analytic. This is parallel to \cite[Theorem 4.6.1]{KL2}The proof is the same as the one where $*$ is just a ring $R$. See \cite[Theorem 4.11]{XT2}.	
\end{proof}

\subsection{Derived Algebraic Geometry of $\infty$-Period Rings}

\indent We now consider the contact with the corresponding algebraic sheaves over the corresponding schemes associated to the corresponding period rings in \cite{KL1}, \cite{KL2} as in \cite{XT1}, \cite{XT2}, \cite{XT3}, \cite{XT4} along \cite{KL1}.

\begin{remark}
Certainly we are now dropping the corresponding topology and functional analyticity. This is motivated by the corresponding discussion on the algebraic quasi-coherent sheaves over schematic Fargues-Fontaine curves in \cite{KL2}.
\end{remark}

\begin{setting}
$A$ is assumed to be now general Banach (commutative at this moment) over the base analytic field, and keep $R$ as in the situation before in this section (in the most general setting we considered so far). Now we apply the whole machinery in \cite{Lu1} to the corresponding $\infty$-Bambozzi-Kremnizer rings $*^h_{R,A}$ attached to the period rings $*_{R,A}$ carrying $A$. Then we regard $*^h_{R,A}$ as a corresponding $\mathbb{E}_1$-ring in \cite{Lu1} \footnote{Note that the original convention of \cite{BK1} is cohomological and vanishing in positive degree.}.
\end{setting}


\begin{proposition}
The corresponding category of all the perfect, almost perfect (after \cite[Section 7.2.4]{Lu1}, namely pseudocoherent) $F^k$-equivariant sheaves over the families of $\infty$-Robba \'etale $\infty$-Deligne-Mumford toposes $(\mathrm{Spec}\widetilde{\Pi}^h_{I,R,A},\mathcal{O})_{\{I \subset (0,\infty)\}}$ (as in \cite[Chapter 1.4]{Lu2}) is equivalent to the category of all the perfect, almost perfect (after \cite[Section 7.2.4]{Lu1}, namely pseudocoherent) $F^k$-equivariant sheaves over some $\infty$-Robba \'etale $\infty$-Deligne-Mumford topos $(\mathrm{Spec}\widetilde{\Pi}^h_{I,R,A},\mathcal{O})$ (as in \cite[Chapter 1.4]{Lu2}) where $I=[r_1,r_2]$ ($0<r_1\leq r_2/p^{hk}$). Here we need to assume the corresponding \cref{conjecture6.22} holds.	
\end{proposition}

\begin{proof}
This is actually quite transparent that we just take the corresponding Frobenius translation to compare the corresponding sheaves.	
\end{proof}

\begin{proposition}
The corresponding category of all the locally finite free (after \cite[Section 7.2.2, 7.2.4]{Lu1}, namely pseudocoherent and flat) $F^k$-equivariant sheaves over the families of $\infty$-Robba \'etale $\infty$-Deligne-Mumford toposes $(\mathrm{Spec}\widetilde{\Pi}^h_{I,R,A},\mathcal{O})_{\{I \subset (0,\infty)\}}$ (as in \cite[Chapter 1.4]{Lu2}) is equivalent to the category of all the locally finite free (after \cite[Section 7.2.2, 7.2.4]{Lu1}, namely pseudocoherent and flat) $F^k$-equivariant sheaves over some $\infty$-Robba \'etale $\infty$-Deligne-Mumford topos $(\mathrm{Spec}\widetilde{\Pi}^h_{I,R,A},\mathcal{O})$ (as in \cite[Chapter 1.4]{Lu2}) where $I=[r_1,r_2]$ ($0<r_1\leq r_2/p^{hk}$).	
\end{proposition}

\begin{proof}
This is again actually quite transparent that we just take the corresponding Frobenius translation to compare the corresponding sheaves.	
\end{proof}

\begin{example}
Everything will be certainly more interesting when we maintain in the corresponding noetherian situation, where we do have the corresponding derived analytic consideration. But since our current goal is to study the corresponding derived algebraic geometry carrying relative $p$-adic Hodge structure we will not explicitly mention the corresponding analytic context.	
\end{example}

\indent Similarly as what we did in the corresponding derived analytic geometry above, we consider the corresponding contact with \cite{GV} in the derived algebraic geometric setting. We keep the notation as above. Then we have the following results with the same arguments as in the above on derived deformation of Hodge structures (note that this is more related to the non-derived situation):

\begin{proposition}
The corresponding category of all the perfect, almost perfect (after \cite[Section 7.2.4]{Lu1}, namely pseudocoherent) $F^k$-equivariant sheaves over the families of $\infty$-Robba \'etale $\infty$-Deligne-Mumford toposes $(\mathrm{Spec}\widetilde{\Pi}_{I,R,A^h},\mathcal{O})_{\{I \subset (0,\infty)\}}$ (as in \cite[Chapter 1.4]{Lu2}) is equivalent to the category of all the perfect, almost perfect (after \cite[Section 7.2.4]{Lu1}, namely pseudocoherent) $F^k$-equivariant sheaves over some $\infty$-Robba \'etale $\infty$-Deligne-Mumford topos $(\mathrm{Spec}\widetilde{\Pi}_{I,R,A^h},\mathcal{O})$ (as in \cite[Chapter 1.4]{Lu2}) where $I=[r_1,r_2]$ ($0<r_1\leq r_2/p^{hk}$). Here we need to assume the corresponding \cref{conjecture6.22} holds.	
\end{proposition}

\begin{proposition}
The corresponding category of all the locally finite free (after \cite[Section 7.2.2, 7.2.4]{Lu1}, namely pseudocoherent and flat) $F^k$-equivariant sheaves over the families of $\infty$-Robba \'etale $\infty$-Deligne-Mumford toposes $(\mathrm{Spec}\widetilde{\Pi}_{I,R,A^h},\mathcal{O})_{\{I \subset (0,\infty)\}}$ (as in \cite[Chapter 1.4]{Lu2}) is equivalent to the category of all the locally finite free (after \cite[Section 7.2.2, 7.2.4]{Lu1}, namely pseudocoherent and flat) $F^k$-equivariant sheaves over some $\infty$-Robba \'etale $\infty$-Deligne-Mumford topos $(\mathrm{Spec}\widetilde{\Pi}_{I,R,A^h},\mathcal{O})$ (as in \cite[Chapter 1.4]{Lu2}) where $I=[r_1,r_2]$ ($0<r_1\leq r_2/p^{hk}$).	
\end{proposition}

\indent The following is expect to hold in full generality.

\begin{conjecture} \label{conjecture6.22}
In the two settings in this subsection above, we conjecture that the corresponding descent holds (along binary rational coverings) for perfect, almost perfect and finite projective modules (which will allow us to compare sheaves and the global sections) in the noetherian situation. \footnote{Any verified $\infty$-descent results in derived algebraic geometry certainly apply.}	
\end{conjecture}

\begin{example}
The corresponding descent could happen when we have the noetherianness, where the corresponding rational localization is actually automatically flat, which will imply that one can certainly glue coherent sheaves in our situation, see \cite[Theorem 1.3.9]{KL1}. Beyond the corresponding noetherianness it might be not safe to conjecture so. 
\end{example}

\newpage

\subsection*{Acknowledgements} 

This is really rooted in our previous works and our previous consideration around the corresponding analytic geometry over relative period rings from \cite{KL1} and \cite{KL2}. We would like to thank Professor Kedlaya for the key discussion around some very subtle points herein.

\newpage

\bibliographystyle{ams}

\begin{thebibliography}{10}
\bibitem[Ked1]{Ked1} Kedlaya, K. "Sheaves, stacks, and shtukas, lecture notes from the 2017 Arizona Winter School: Perfectoid Spaces." Math. Surveys and Monographs 242.

\bibitem[KL1]{KL1} Kedlaya, Kiran Sridhara, and Ruochuan Liu. Relative $p$-adic Hodge theory: foundations. Soci\'et\'e math\'ematique de France, 2015.

\bibitem[KL2]{KL2} Kedlaya, Kiran S., and Ruochuan Liu. "Relative $p$-adic Hodge theory, II: Imperfect period rings." arXiv preprint arXiv:1602.06899 (2016).

\bibitem[GR]{GR} Gabber, Ofer, and Lorenzo Ramero. "Foundations for almost ring theory--Release 7.5." arXiv preprint math/0409584 (2004).

\bibitem[XT1]{XT1} Tong, Xin. "Hodge-Iwasawa Theory I." arXiv preprint arXiv:2006.03692 (2020).

\bibitem[XT2]{XT2} Tong, Xin. "Hodge-Iwasawa Theory II." arXiv preprint arXiv:2010.06093 (2020).

\bibitem[XT3]{XT3} Tong, Xin. "Period Rings with Big Coefficients and Applications I." arXiv preprint arXiv:2012.07338 (2020).

\bibitem[XT4]{XT4} Tong, Xin. "Period Rings with Big Coefficients and Application II." arXiv preprint arXiv:2101.03748 (2021).




\bibitem[BK]{BK1}  Bambozzi, Federico, and Kobi Kremnizer. "On the Sheafyness Property of Spectra of Banach Rings." arXiv preprint arXiv:2009.13926 (2020).

\bibitem[CS]{CS} Dustin, Clausen, and Peter Scholze. Lectures on Analytic Geometry. https://www.math.uni-bonn.de/people/scholze/Notes.html. 


\bibitem[Lu1]{Lu1} Lurie, Jacob. "Higher algebra. 2014." Preprint, available at http://www. math. harvard. edu/~ lurie (2016).

\bibitem[Lu2]{Lu2} Lurie, Jacob. "Spectral algebraic geometry." unpublished paper (2018).



\bibitem[GV]{GV} Galatius, Soren, and Akshay Venkatesh. "Derived Galois deformation rings." Advances in Mathematics 327 (2018): 470-623.

\bibitem[BBBK]{BBBK} Bambozzi, Federico, Oren Ben-Bassat, and Kobi Kremnizer. "Analytic geometry over $\mathbb{F}_1$ and the Fargues-Fontaine curve." Advances in Mathematics 356 (2019): 106815.

\end{thebibliography}

\end{document}